\documentclass[12pt,a4paper]{article}
\usepackage{amsfonts}
\usepackage{amssymb}
\usepackage{mathrsfs}
\usepackage{amsmath}
\usepackage{amsthm}
\usepackage{enumerate}
\usepackage{dsfont}

\allowdisplaybreaks
\newtheorem{defn}{Definition}[section]
\newtheorem{thm}[defn]{Theorem}

\newtheorem{prop}[defn]{Proposition}
\newtheorem{cor}[defn]{Corollary}
\newtheorem{ex}[defn]{Example}
\newtheorem{re}[defn]{Remark}
\begin{document}
\title{{\bf Product and complex
structures on 3-Bihom-Lie algebras
}}
\author{\normalsize \bf  Juan Li, Ying Hou, Liangyun Chen}
\date{{{\small{ School of Mathematics and Statistics, Northeast Normal University,  \\Changchun 130024,  China
 }}}} \maketitle

\date{}
\maketitle
\begin{abstract}
In this paper, we first introduce the notion of a product structure on a $3$-Bihom-Lie algebra which is a Nijenhuis operator with some
conditions. And we provide that a $3$-Bihom-Lie algebra has a product structure if and only if it is the direct sum of two vector spaces which are also Bihom subalgebras. Then we give four special conditions such that each of them can make the $3$-Bihom-Lie algebra has a special decomposition. Similarly, we introduce the definition of a complex structure
on a $3$-Bihom-Lie algebra and there are also four types special complex structures. Finally, we
show that the relation between a complex structure and a product structure.

\end{abstract}
\renewcommand{\thefootnote}{\fnsymbol{footnote}}
\footnote[0]{Corresponding author(L. Chen): chenly640@nenu.edu.cn.}
\footnote[0]{Supported by  NNSF of China (Nos. 11771069 and 12071405)}

\section{Introduction}

In \cite{GMMP}, the authors introduced the notion of Bihom-algebras when they studied generalized Hom-structures. There are two twisted maps $\alpha$ and $\beta$ in this class of algebras. Bihom-algebras are Hom-algebras with $\alpha=\beta$. And if $\alpha=\beta=\rm Id$, Bihom-algebras will be return to algebras. Since then many authors are interesting in Bihom-algebras, such as Bihom-Lie algebras \cite{CQ}, Bihom-Lie superalgebras \cite{JLY}, Bihom-Lie colour algebras \cite{KAA}, BiHom-bialgebras \cite{LMC} and so on. In particular, the authors given the definition of $n$-Bihom-Lie algebras and $n$-Bihom-Associative algebras in \cite{AAS}. Considering $n=3$, the authors given some constructions of $3$-Bihom-Lie algebras in \cite{LC}. Moreover, in \cite{BN} the notion of Nijenhuis operators on $3$-BiHom-Lie superalgebras were given, and the authors studied the relation between Nijenhuis operators and Rota-Baxter operators.

Product structures and complex structures with some conditions are Nijenhuis operators. Many authors studied product structures and complex structures on Lie algebras from many points of view in \cite{AA,AB1,AB2,AS,SS}. In particular, the authors have studied product structures and complex structures on 3-Lie algebras in \cite{ST}. And then product structures and complex structures on involution Hom-3-Lie algebras have been studied in \cite{QLJ}. Based on these facts, we want to consider the product structures and complex structures on 3-Bihom-Lie algebras.

This paper is organized as follows. In Section 2, we recall the definition of  3-Bihom-Lie algebras
and Nijenhuis operators on 3-Bihom-Lie algebras. In Section 3, we introduce the notion of a product structure on 3-Bihom-Lie algebra, and obtain that the necessary and sufficient condition of a 3-Bihom-Lie algebra having a product structure is that it is the direct sum of two Bihom-subalgebras (as vector spaces). Then we give four special product structures. In Section 4, we introduce the notion of a complex structure on a 3-Bihom-Lie
algebra and also give four special conditions. And then we introduce the notion of a complex product structure on 3-Bihom-Lie algebra which is a product structure and a complex structure with a condition. We give the necessary and sufficient condition of a 3-Bihom-Lie algebra having a complex product structure.

\section{Preliminary}

\begin{defn}\cite{AAS}
A $3$-Bihom-Lie algebra over a field $\mathbb{K}$ is a
$4$-tuple $(L,[\cdot,\cdot,\cdot],\alpha,\beta)$, which consists of a vector
space $L$, a $3$-linear operator $[\cdot,\cdot,\cdot]
\colon L\times L\times L\rightarrow L$ and two linear maps $\alpha,\beta \colon L\rightarrow L$, such that $\forall\, x,y,z,u,v\in L$,
\begin{enumerate}[(1)]
\item $\alpha\beta=\beta\alpha$,
\item $\alpha ([x,y,z])=[\alpha(x),\alpha(y),\alpha(z)]$ and
  $\beta([x,y,z])=[\beta(x) ,\beta(y),\beta(z)]$,
\item Bihom-skewsymmetry: $[\beta(x),\beta(y),\alpha(z)]=-[\beta(y),\beta(x),\alpha(z)]=-[\beta(x),\beta(z),\alpha(y)],$
\item $3$-BiHom-Jacobi identity:
\begin{eqnarray*}
 &&[\beta^{2}(u),\beta^{2}(v),[\beta(x),\beta(y),\alpha(z)]]\\
 &=&[\beta^{2}(y),\beta^{2}(z),[\beta(u),\beta(v),\alpha(x)]]-[\beta^{2}(x),\beta^{2}(z),[\beta(u),\beta(v),\alpha(y)]]\\
 & &+[\beta^{2}(x),\beta^{2}(y),[\beta(u),\beta(v),\alpha(z)]].
\end{eqnarray*}
\end{enumerate}
\end{defn}
A $3$-Bihom-Lie algebra is called a \textit{regular} $3$-Bihom-Lie algebra if $\alpha$ and $\beta$ are algebra automorphisms.

\begin{defn}\cite{AAS}
 A sub-vector space $\eta\subseteq L$ is a Bihom subalgebra of $(L,[\cdot,\cdot, \cdot],\alpha,\beta)$ if $\alpha(\eta)\subseteq \eta$, $\beta(\eta)\subseteq \eta$ and $[x,y,z]\in \eta,~\forall\, x,y,z\in \eta$.
  It is said to be a Bihom ideal of $(L,[\cdot,\cdot,\cdot],\alpha,\beta)$ if $\alpha(\eta)\subseteq\eta$, $\beta(\eta)\subseteq \eta$ and $[x,y,z]\in \eta,~\forall\, x\in \eta,y,z\in L.$
\end{defn}

\begin{defn}\cite{BN}
Let $(L, [\cdot, \cdot, \cdot],\alpha,\beta)$ be a $3$-Bihom-Lie algebra. A linear map $N : L\rightarrow L$ is called
a Nijenhuis operator if for all $x, y, z \in L$, the following conditions hold :
\begin{align*}
\alpha N=& N\alpha,~\beta N=N\beta,\\
[N x, Ny, Nz]=& N[N x, Ny, z]+N[N x, y, Nz]+N[x, Ny, Nz]\\
&-N^{2}[ Nx, y, z]-N^{2}[x, Ny, z]-N^{2}[x, y, Nz]+N^{3}[x, y, z].
\end{align*}
\end{defn}

\section{Product structures on $3$-Bihom-Lie algebras}
In this section, we introduce the notion of a product structure on a 3-Bihom-Lie algebra by the Nijenhuis
opeartors. We find four special conditions, and each of them
gives a special decomposition of the $3$-Bihom-Lie algebra. At the end of this section, we give some examples.
\begin{defn}
Let $(L, [\cdot, \cdot, \cdot],\alpha,\beta)$ be a 3-Bihom-Lie algebra.
An almost product structure
on $L$ is a linear map $E : L\rightarrow L$ satisfying $E^2 = \rm Id$~$(E \neq\pm\rm Id)$, $\alpha E=E \alpha$ and $\beta E=E \beta$.

An almost product structure is
called a product structure if $\forall x, y, z\in L,$
\begin{align}\label{21}
E[x, y,x] =& [Ex,Ey,Ez]+[Ex,y,z]+[x,Ey,z]+[x,y,Ez]\nonumber\\
 &-E[Ex,Ey,z]-E[x,Ey,Ez]-E[Ex,y,Ez].
\end{align}
\end{defn}
\begin{re}
If $E$ is a Nijenhuis operator on a 3-Bihom-Lie algebra with $E^2 = \rm Id$. Then $E$ is a product structure.
\end{re}

\begin{thm}\label{thm1}
Let $(L, [\cdot, \cdot,\cdot],\alpha,\beta)$ be a 3-Bihom-Lie algebra. Then there is a product structure on $L$ if and only if $L=L_+\oplus L_-,$
where $L_+$ and $L_-$ are Bihom subalgebras of $L$.
\end{thm}

\begin{proof}
Let $E$ be a product structure on $L$. By $E^2 = \rm Id$,  we have that $L=L_+\oplus L_-$, where $L_+$ and $L_-$ are eigenspaces of $L$ associated to the eigenvalues 1 and -1 respectively, i.e., $L_+=\{x\in L |Ex=x\},~ L_-=\{x\in L |Ex=-x\}$. For all $x_1, x_2, x_3\in L_+$, we can show
\begin{align*}
E[x_1, x_2, x_3]&=[Ex_1, Ex_2, Ex_3]+[Ex_1, x_2, x_3]+[x_1, Ex_2, x_3]+[x_1, x_2, Ex_3]\\
&\quad -E[Ex_1, Ex_2, x_3]-E[x_1, Ex_2, Ex_3]-E[Ex_1, x_2, Ex_3]\\
&= 4[x_1, x_2, x_3]-3E[x_1, x_2, x_3].
\end{align*}
Thus, we have $E[x_1, x_2, x_3] = [x_1, x_2, x_3]$, which implies that $ [x_1, x_2, x_3]\in L_+$. And since $E\alpha(x_1)=\alpha E(x_1)=\alpha(x_1)$, we can get $\alpha(x_1)\in L_+$. Similarly, we have $\beta(x_1)\in L_+$. So $L_+$ is a Bihom subalgebra. By the same way, we can obtain that $L_-$ is also a Bihom subalgebra.

Conversely, a linear map $E : L \rightarrow L$ defined by
\begin{equation}\label{E defin}
E(x + y) = x-y, ~\forall x\in L_+, ~y\in L_- .
\end{equation}
Obviously, $E^2 =\rm Id$. Since $L_+$ and $L_-$ are Bihom subalgebras of $L$, for all $x\in L_+, y \in L_-$, we have
\begin{equation*}
E\alpha(x+y)=E(\alpha(x)+\alpha(y))=\alpha(x)-\alpha(y)=\alpha(x-y)=\alpha E(x-y),
\end{equation*}
which implies that $E\alpha=\alpha E$. Similarly, we have $E\beta=\beta E$. In addition, for all $x_i\in L_+,~y_j\in L_-,~i,j=1,2,3$, we can obtain
\begin{align*}
&[E(x_1+y_1), E(x_2+y_2), E(x_3+y_3)]+[E(x_1+y_1), x_2+y_2, x_3+y_3]\\
&+[x_1+y_1, E(x_2+y_2), x_3+y_3]+[x_1+y_1, x_2+y_2, E(x_3+y_3)]\\
&-E([E(x_1+y_1), E(x_2+y_2), x_3+y_3]+[E(x_1+y_1), x_2+y_2, E(x_3+y_3)]\\
&+[x_1+y_1, E(x_2+y_2), E(x_3+y_3)])\\
=&[x_1-y_1, x_2-y_2, x_3-y_3]+[x_1-y_1, x_2+y_2, x_3+y_3]+[x_1+y_1, x_2-y_2, x_3+y_3]\\
&+[x_1+y_1, x_2+y_2, x_3-y_3]-E([x_1-y_1, x_2-y_2, x_3+y_3]+[x_1-y_1, x_2+y_2, x_3-y_3]\\
&+[x_1+y_1, x_2-y_2, x_3-y_3])\\
=&4[x_1,x_2,x_3]-4[y_1,y_2,y_3]-E(3[x_1,x_2,x_3]-[x_1,x_2,y_3]-[x_1,y_2,x_3]-[x_1,y_2,y_3])\\
&-[y_1,x_2,x_3]-[y_1,x_2,y_3]-[y_1,y_2,x_3]+[y_1,y_2,y_3])\\
=&E([x_1,x_2,x_3]+[x_1,x_2,y_3]+[x_1,y_2,x_3]+[x_1,y_2,y_3]+[y_1,x_2,x_3]+[y_1,x_2,y_3]\\
&+[y_1,y_2,x_3]+[y_1,y_2,y_3])\\
=&E([x_1+y_1, x_2+y_2, x_3+y_3]).
\end{align*}
Therefore, $E$ is a product structure on $L$.
\end{proof}

\begin{prop}\label{prop1}
Let $(L, [\cdot,\cdot,\cdot],\alpha,\beta)$ be a $3$-Bihom-Lie algebra with $\alpha,\beta$ surjective and $E$ be an almost product structure on $L$.
If $E$ satisfies the following equation
\begin{equation}\label{23}
E[x, y,z] = [Ex, y,z],~\forall x,y,z\in L,
\end{equation}
then $E$ is a product structure on $L$ such that
$[L_+,L_+ L_-] = [L_-, L_-, L_+]=0$, i.e.,
$L$ is the direct sum of $L_+$ and $L_-$.
\end{prop}
\begin{proof}
By $(\ref{23})$ and $E^2 =\rm Id$, we have
\begin{align*}
&[Ex, Ey, Ez]+[Ex, y, z]+[x, Ey, z]+ [x, y, Ez]\\
&-E[Ex, Ey, z]-E[x, Ey, Ez]-E[Ex, y, Ez]\\
= &[Ex, Ey, Ez]+E[x, y, z]+[x, Ey, z] +[x, y, Ez]\\
&-[E^2x, Ey, z]-[Ex, Ey, Ez]-[E^2x, y, Ez]\\
=& E[x, y, z].
\end{align*}
Thus $E$ is a product structure on $L$. By Theorem \ref{thm1}, we have $L=L_+\oplus L_-$, where $L_+$ and $L_-$ are Bihom subalgebras. For all $x_1, x_2 \in L_+, x_3\in L_-$, on one hand we have
$$E[x_1, x_2, x_3] = [Ex_1, x_2, x_3] = [x_1, x_2, x_3].$$ On the other hand. Since $\alpha,\beta$ are surjective, we have $\tilde{x}_1,\tilde{x}_2\in L_+, \tilde{x}_3\in L_-$ such taht $x_1=\beta(\tilde{x}_1),x_2=\beta(\tilde{x}_2)$ and $x_3=\alpha(\tilde{x}_3)$. So we can get
\begin{align*}
&E[x_1, x_2, x_3] = E[\beta(\tilde{x}_1),\beta(\tilde{x}_2), \alpha(\tilde{x}_3)]=E[\beta(\tilde{x}_3),\beta(\tilde{x}_1), \alpha(\tilde{x}_2)]\\
=&[E\beta(\tilde{x}_3),\beta(\tilde{x}_1), \alpha(\tilde{x}_2)]
=-[\beta(\tilde{x}_3),\beta(\tilde{x}_1), \alpha(\tilde{x}_2)]=-[\beta(\tilde{x}_1),\beta(\tilde{x}_2), \alpha(\tilde{x}_3)]\\
=&-[x_1, x_2, x_3].
\end{align*}
Thus, we obtain $[L_+,L_+,L_-] = 0$. Similarly, we have $[L_-,L_-,L_+] = 0$. The proof is finished.
\end{proof}

\begin{defn}
An almost product structure $E$ on a $3$-Bihom Lie algebra $(L, [\cdot, \cdot,\cdot], \alpha,\beta)$ is called
a strict product structure if $(\ref{23})$ holds.
\end{defn}

\begin{cor}
Let $(L, [\cdot, \cdot,\cdot], \alpha,\beta)$ be a $3$-Bihom Lie algebra with $\alpha,\beta$ surjective. Then $L$ has a strict product structure if and only if $L = L_+\oplus L_- $,
where $L_+$ and $L_-$ are Bihom subalgebras of $L$ such that $[L_+,L_+,L_-] = 0$ and $[L_-,L_-,L_+] = 0$.
\end{cor}
\begin{proof}
 Let $E$ be a strict product structure on $L$. By Proposition \ref{prop1} and Theorem \ref{thm1}, we can obtain $L = L_+\oplus L_- $,
where $L_+$ and $L_-$ are Bihom subalgebras of $L$ such that $[L_+,L_+,L_-] = 0$ and $[L_-,L_-,L_+] = 0$.

Conversely, we can define a linear endomorphism $E$ by $(\ref{E defin})$. Since $\alpha,\beta$ are surjective and $[L_+,L_+,L_-] =[L_-,L_-,L_+] = 0$. So we have $[L_+,L_-,L_+] =[L_-,L_+,L_+]=[L_+,L_-,L_-] =[L_-,L_+,L_-] = 0$. For all $x_i\in L_+, y_j\in L_-,i,j=1,2,3$, we can show
\begin{align*}
&E[x_1+y_1,x_2+y_2,x_3+y_3]\\
=&E([x_1,x_2,x_3]+[y_1,y_2,y_3])\\
=&[x_1,x_2,x_3]-[y_1,y_2,y_3]\\
=&[x_1-y_1,x_2+y_2,x_3+y_3]\\
=& [E(x_1+y_1),x_2+y_2,x_3+y_3].
\end{align*}
Thus $E$ is a strict product structure on $L$.
\end{proof}

\begin{prop}\label{3.7}
Let $(L, [\cdot, \cdot,\cdot], \alpha,\beta)$ be a $3$-Bihom Lie algebra and $E$ be an almost product structure on $L$.
If $E$ satisfies the following equation
\begin{equation}\label{26}
[x, y, z] = -[x, Ey, Ez]-[Ex, y, Ez]- [Ex, Ey, z], \forall x,y,z\in L,
\end{equation}
then $E$ is a product structure on $L$.
\end{prop}

\begin{proof}
Using $(\ref{26})$ and $E^2 = \rm Id$, we have
\begin{align*}
&[Ex, Ey, Ez]+[Ex, y, z]+[x, Ey, z]+ [x, y, Ez]\\
&-E[Ex, Ey, z]- E[x, Ey, Ez]-E[Ex, y, Ez]\\
=&-[Ex, E^2y, E^2z]-[E^2x, Ey, E^2z]- [E^2x, E^2y, Ez]\\
&+[Ex, y, z]+[x, Ey, z]+[x, y, Ez]+E[x, y, z]\\
=& E[x, y, z].
\end{align*}
Thus $E$ is a product structure on $L$.
\end{proof}

\begin{defn}
An almost product structure $E$ on a $3$-Bihom Lie algebra $(L, [\cdot, \cdot,\cdot], \alpha,\beta)$ is called
an abelian product structure if $(\ref{26})$ holds.
\end{defn}

\begin{cor}\label{cor3.9}
Let $(L, [\cdot, \cdot,\cdot], \alpha,\beta)$ be a $3$-Bihom Lie algebra. Then there is an abelian product
structure on $L$ if and only if $
L=L_+\oplus L_-,
$
where $L_+$ and $L_-$ are abelian Bihom subalgebras of $L$.
\end{cor}
\begin{proof}
Let $E$ be an abelian product structure on $L$. By Theorem \ref{thm1} and Propsitoon \ref{3.7}, we only need to show that $L_+$ and $L_-$ are abelian Bihom subalgebras of $L$. For all $x_1, x_2, x_3 \in L_+$, we have
\begin{align*}
[x_1, x_2, x_3]&=-[Ex_1, Ex_2, x_3]-[x_1, Ex_2, Ex_3]-[Ex_1, x_2, Ex_3]\\
&=-3[x_1, x_2, x_3],
\end{align*}
which implies that $[x_1, x_2, x_3]= 0$. Similarly, $[y_1, y_2, y_3]= 0, \forall y_1, y_2, y_3 \in L_-$.
Thus, $L_+$ and $L_-$ are abelian Bihom subalgebras.

Conversely, define a linear endomorphism $E : L\rightarrow L$ by $(\ref{E defin})$. Then for all $ x_i\in L_+, y_j\in L_-,i,j=1,2,3$,
\begin{align*}
&-[x_1+y_1, E(x_2+y_2), E(x_3+y_3)]-[E(x_1+y_1), E(x_2+y_2), x_3+y_3]\\
&-[E(x_1+y_1), x_2+y_2, E(x_3+y_3)]\\
=&-[x_1+y_1, x_2-y_2, x_3-y_3]-[x_1-y_1, x_2-y_2, x_3+y_3]-[x_1-y_1, x_2+y_2, x_3-y_3]\\
=&[x_1,x_2,y_3]+[x_1,y_2,x_3]+[x_1,y_2,y_3]+[y_1,x_2,x_3]+[y_1,x_2,y_3]+[y_1,y_2,x_3]\\
=&[x_1+y_1, x_2+y_2, x_3+y_3],
\end{align*}
i.e., $E$ is an abelian product structure on $L$.
\end{proof}

\begin{prop}\label{re}
Let $E$ be an almost product structure on a $3$-Bihom Lie algebra $(L, [\cdot, \cdot,\cdot], \alpha,\beta)$ .
If $E$ satisfies the following equation, for all $x,y,z\in L$,
\begin{equation}\label{30}
[x, y, z] = E[Ex, y, z]+E[x, Ey, z]+E[x, y, Ez],
\end{equation}
then $E$ is an abelian product structure on $L$ such that
$[L_+, L_+ ,L_-]\subseteq L_+,
[L_-, L_- ,L_+]\subseteq L_-.$
\end{prop}
\begin{proof}
By $(\ref{30})$ and $E^2 = \rm Id$ we have
\begin{align*}
&[Ex, Ey, Ez]+[Ex, y, z]+[x, Ey, z]+ [x, y, Ez]\\
&-E[Ex, Ey, z]-E[x, Ey, Ez]-E[Ex, y, Ez]\\
=&E[x, Ey, Ez]+E[Ex, y, Ez]+E[Ex, Ey, z]+E[x, y, z]\\
&-E[Ex, Ey, z]-E[x, Ey, Ez]-E[Ex, y, Ez]\\
=&E[x, y, z].
\end{align*}
Thus, we obtain that $E$ is a product structure on $L$. For all $x_1, x_2, x_3 \in L_+$, by $(\ref{30})$, we have
\begin{align*}
[x_1, x_2, x_3]&= E[Ex_1, x_2, x_3]+ E[x_1, Ex_2, x_3]+E[x_1, x_2, Ex_3]\\
&=3E[x_1, x_2, x_3]=3[x_1, x_2, x_3].
\end{align*}
So we get $[L_+, L_+, L_+] = 0$. Similarly, we have $[L_-, L_-, L_-] = 0$. By Corollary \ref{cor3.9}, $E$ is an
abelian product structure on $L$.
Moreover, for all $x_1, x_2 \in L_+, y_1 \in L_-$, we have
\begin{align*}
[x_1, x_2, y_1]&= E[Ex_1, x_2, y_1]+ E[x_1, Ex_2, y_1]+E[x_1, x_2, Ey_1]\\
&= E[x_1, x_2, y_1],
\end{align*}
which implies that $[L_+, L_+ ,L_-]\subseteq L_+$. Similarly, we have $[L_-, L_- ,L_+]\subseteq L_-$.
\end{proof}
\begin{re}
In Proposition \ref{re} we can also obtain that $[L_+, L_- ,L_+]\subseteq L_+$, $[L_-, L_+ ,L_+]\subseteq L_+$,  $[L_-, L_+ ,L_-]\subseteq L_-$,  $[L_+, L_- ,L_-]\subseteq L_-$.
\end{re}

\begin{defn}
An almost product structure $E$ on a $3$-Bihom Lie algebra $(L, [\cdot, \cdot,\cdot], \alpha,\beta)$ is called
a strong abelian product structure if $(\ref{30})$ holds.
\end{defn}

\begin{cor}
Let $(L, [\cdot, \cdot,\cdot], \alpha,\beta)$ be a $3$-Bihom Lie algebra. Then there is a strong abelian product structure on $L$ if and only if $
L= L_+\oplus L_- ,
$
where $L_+$ and $L_-$ are abelian Bihom subalgebras of $L$ such that
$[L_+, L_+ ,L_-]\subseteq L_+$, $[L_+, L_- ,L_+]\subseteq L_+, [L_-, L_+ ,L_+]\subseteq L_+, [L_-, L_+ ,L_-]\subseteq L_-, [L_+, L_- ,L_-]\subseteq L_-$ and $[L_-, L_-,L_+]\subseteq L_-$.
\end{cor}

\begin{prop}\label{3.13}
Let $E$ be an almost product structure on a $3$-Bihom-Lie algebra $(L, [\cdot, \cdot,\cdot], \alpha,\beta)$. If $E$ satisfies
the following equation
\begin{equation}\label{31}
E[x, y, z] = [Ex, Ey, Ez],
\end{equation}
then $E$ is a product structure on $L$ such that
$[L_+, L_+ ,L_-]\subseteq L_-$, $[L_-, L_- ,L_+]\subseteq L_+$.
\end{prop}
\begin{proof}
By $(\ref{31})$ and $E^2 = \rm Id$ we have
\begin{align*}
&[Ex, Ey, Ez]+[Ex, y, z]+[x, Ey, z]+ [x, y, Ez]\\
&-E[Ex, Ey, z]-E[x, Ey, Ez]-E[Ex, y, Ez]\\
=&E[x, y, z]+[Ex, y, z]+[x, Ey, z]+[x, y, Ez]\\
&-[E^2x, E^2y, Ez]-[Ex, E^2y, E^2z]-[E^2x, Ey, E^2z]\\
=&E[x, y, z].
\end{align*}
Thus, $E$ is a product structure on $L$. Moreover, for all $x_1, x_2 \in L_+, y_1 \in L_-$, we have
\begin{align*}
E[x_1, x_2, y_1] = [Ex_1, Ex_2, Ey_1] = -[x_1, x_2, y_1],
\end{align*}
which implies that $[L_+, L_+ ,L_-]\subseteq L_-$. Similarly, we have $[L_-, L_- ,L_+]\subseteq L_+$.
\end{proof}

\begin{re}
In Proposition \ref{3.13} we can also obtain that $[L_+, L_- ,L_+]\!\subseteq\! L_-$, $[L_-, L_+ ,L_+]$ $\subseteq L_-$,  $[L_-, L_+ ,L_-]\subseteq L_+$, $[L_+, L_- ,L_-]\subseteq L_+$
\end{re}

\begin{defn}
An almost product structure $E$ on a $3$-Bihom Lie algebra $(L, [\cdot, \cdot,\cdot], \alpha,\beta)$ is called
a perfect product structure if $(\ref{31})$ holds.
\end{defn}

\begin{cor}
Let $(L, [\cdot, \cdot,\cdot], \alpha,\beta)$ be a $3$-Bihom Lie algebra. Then $L$ has a perfect product structure if and only if $
L= L_+\oplus L_- ,
$
where $L_+$ and $L_-$ are Bihom subalgebras of $L$ such that
$
[L_+, L_+ ,L_-]\subseteq L_-$, $[L_+, L_- ,L_+]\subseteq L_-$, $[L_-, L_+ ,L_+]\subseteq L_-$, $[L_-, L_-,L_+]\subseteq L_+$, $[L_-, L_+ ,L_-]\subseteq L_+$ and $[L_+, L_- ,L_-]\subseteq L_+$.
\end{cor}

\begin{cor}
A strict product structure on a $3$-Bihom-Lie algebra is a perfect product structure.
\end{cor}

\begin{ex}\label{ex}
Let $L$ be a $3$-dimensional vector space with respect to a basis $\{e_1, e_2, e_3 \}$, the non-zero bracket and $\alpha,\beta$ be given by
\begin{equation*}
[e_1, e_2,e_3] = [e_1, e_3,e_2] = [e_2, e_3,e_1] = e_2,~[e_2, e_1,e_3] = [e_3, e_1,e_2] = [e_3, e_2,e_1] =- e_2,
\end{equation*}
$$\alpha=\rm{Id},~ \beta=\left(                 
  \begin{array}{ccc}   
    -1 & 0 & 0 \\  
    0 & 1 & 0 \\
    0 & 0 & -1
  \end{array}
\right).$$ Then $(L, [\cdot, \cdot,\cdot], \alpha,\beta)$ is a $3$-Bihom Lie algebra.
So
$ E=\left(                 
  \begin{array}{ccc}   
    -1 & 0 & 0 \\  
    0 & 1 & 0 \\
    0 & 0 & -1
  \end{array}
\right)$
is a perfect product structure and an abelian product structure, $ E=\left(                 
  \begin{array}{ccc}   
    -1 & 0 & 0 \\  
    0 & 1 & 0 \\
    0 & 0 & 1
  \end{array}
\right)$ is a strong abelian product structure and $ E=\left(                 
  \begin{array}{ccc}   
    1 & 0 & 0 \\  
    0 & 1 & 0 \\
    0 & 0 & -1
  \end{array}
\right)$ is a strong abelian product structure and a strict product structure.
\end{ex}

\begin{ex}\label{ex1}
Let $L$ be a $4$-dimensional $3$-Lie algebra with respect to a basis $\{e_1, e_2, e_3,e_4\}$ satisfied the bracket
\begin{equation*}
[e_1, e_2,e_3] = e_4,~[e_2, e_3,e_4] = e_1,~[e_1, e_3,e_4] = e_2,~[e_1, e_2,e_4] = e_3.
\end{equation*}
We can define maps $\alpha,\beta:L\rightarrow L$ by
$       
\alpha=\left(                 
  \begin{array}{cccc}   
    -1 & 0 & 0 & 0 \\  
    0 & 1 & 0 & 0 \\
    0 & 0 & 1 & 0 \\
    0 & 0 & 0 & -1
  \end{array}
\right)
$ and $       
\beta=\left(                 
  \begin{array}{cccc}   
    -1 & 0 & 0 & 0 \\  
    0 & -1 & 0 & 0 \\
    0 & 0 & 1 & 0 \\
    0 & 0 & 0 & 1
  \end{array}
\right).
$               
Then $(L, [\cdot, \cdot,\cdot], \alpha,\beta)$ is a $3$-Bihom Lie algebra. Thus
\begin{equation*}       
E=\left(                 
  \begin{array}{cccc}   
    -1 & 0 & 0 & 0\\  
    0 & -1 & 0 & 0\\
    0 & 0 & 1 & 0\\
    0 & 0 & 0 & 1
  \end{array}
\right),
E=\left(                 
  \begin{array}{cccc}   
    -1 & 0 & 0 & 0\\  
    0 & 1 & 0 & 0\\
    0 & 0 & 1 & 0\\
    0 & 0 & 0 & -1
  \end{array}
\right),
E=\left(                 
  \begin{array}{cccc}   
    -1 & 0 & 0 & 0\\  
    0 & 1 & 0 & 0\\
    0 & 0 & -1 & 0\\
    0 & 0 & 0 & 1
  \end{array}
\right)                 
\end{equation*}
are perfect and abelian product structures on $L$.
\end{ex}

\section{Complex structures on $3$-Bihom-Lie algebras}
In this section, we introduce the notion of a complex structure on a real $3$-Bihom-Lie algebra. There are also some special complex structures which parallel to the case of product structures.
\begin{defn}\label{4.1}
Let $(L, [\cdot, \cdot,\cdot],\alpha,\beta)$ be a $3$-Bihom-Lie algebra. An almost complex structure
on $L$ is a linear map $J : L\rightarrow L$ satisfying $J^2 = -\rm Id$, $J\alpha=\alpha J$ and $J\beta=\beta J$. An almost complex structure is
called a complex structure if the following condition holds$:$
\begin{align}\label{eq4.1}
J[x, y, z]=&-[Jx, Jy, Jz]+[Jx, y, z]+[x, Jy, z]+[x, y, Jz]\nonumber\\
&+J[Jx, Jy, z]+J[x, Jy, Jz]+J[Jx, y, Jz].
\end{align}
\end{defn}

\begin{re}
A complex structure $J$ on a $3$-Bihom-Lie algebra $L$ means that $J$ is a Nijenhuis operator
satisfying $J^2 = -\rm Id$.
\end{re}

\begin{re}
We can also use Definition \ref{4.1} to give the notion of a complex structure on a
complex $3$-Bihon-Lie algebra with considering $J$ to be $\mathds{C}$-linear. However, this is not very interesting. Because
for a complex $3$-Bihom-Lie algebra there is a one-to-one correspondence between $\mathds{C}$-linear complex
structures and product structures (see Proposition $\ref{4.23}$).
\end{re}

Now we consider that $L_{\mathds{C}} = L\otimes_{\mathds{R}}\mathds{C} = \{x + iy | x, y \in L\}$ is the complexification of a real $3$-Bihom-Lie algebra $L$. Obviously $(L_{\mathds{C}}, [\cdot, \cdot,\cdot]_{L_{\mathds{C}}},\alpha_{\mathds{C}},\beta_{\mathds{C}})$ is a $3$-Bihom-Lie algebra, where $[\cdot, \cdot,\cdot]_{L_{\mathds{C}}}$ is defined by extending the bracket on $L$ complex trilinearly, and $\alpha_{\mathds{C}}(x+iy)=\alpha(x)+i\alpha(y),\beta_{\mathds{C}}(x+iy)=\beta(x)+i\beta(y)$ for all $x,y\in L$. Let $\sigma$ be the conjugation map in $L_{\mathds{C}}$, that is
$
\sigma(x + iy) = x-iy, ~\forall x, y \in L.
$ Then, $\sigma$ is a complex antilinear, involutive automorphism of the
complex vector space $L_{\mathds{C}}$.

\begin{re}\label{re4.4}
Let $(L, [\cdot, \cdot,\cdot],\alpha,\beta)$ be a real $3$-Bihom-Lie algebra and $J$ be a complex structure on $L$. We extend the complex structure $J$ complex linearly, which is denoted $J_{\mathds{C}}$, i.e., $J_{\mathds{C}} : L_{\mathds{C}} \rightarrow L_{\mathds{C}}$ is defined by
\begin{equation*}
J_{\mathds{C}}(x + iy) = Jx + iJy, ~\forall x, y \in L.
\end{equation*}
Then $J_{\mathds{C}}$ is a complex linear endomorphism on $L_{\mathds{C}}$ such that  $J_{\mathds{C}}\alpha_{\mathds{C}}=\alpha_{\mathds{C}} J_{\mathds{C}}$, $J_{\mathds{C}}\beta_{\mathds{C}}=\beta_{\mathds{C}}J_{\mathds{C}}$,  $J_{\mathds{C}}^2=-{\rm {Id}}_{L_{\mathds{C}}}$ and $(\ref{eq4.1})$ holds, i.e., $J_{\mathds{C}}$ is a complex structure on $L_{\mathds{C}}$.
\end{re}

\begin{thm}\label{thm4}
Let $(L, [\cdot, \cdot,\cdot],\alpha,\beta)$ be a real $3$-Bihom-Lie algebra. Then there is a complex structure on $L$ if and only if $
L_{\mathds{C}}=Q\oplus P,
$
where $Q$ and $P=\sigma(Q)$ are Bihom subalgebras of $L_{\mathds{C}}$.
\end{thm}
\begin{proof}
Suppose $J$ is a complex structure on $L$. By Remark \ref{re4.4} $J_{\mathds{C}}$ is a complex structure on $L_{\mathds{C}}$.
Denote by $L_{\pm i}$ the corresponding eigenspaces of $L_{\mathds{C}}$ associated to the eigenvalues $\pm i$
and so $L_{\mathds{C}} = L_i \oplus L_{-i}$. We can also get
\begin{align*}
L_i &= \{x\in L_{\mathds{C}}|J_{\mathds{C}}(x)=ix\}=\{x-iJx |x\in L\},\\
L_{-i}&= \{x\in L_{\mathds{C}}|J_{\mathds{C}}(x)=ix\} = \{x + iJx |x\in L\}.
\end{align*}
So we have $L_{-i}=\sigma(L_i)$, $\alpha_{\mathds{C}}(L_i)\subseteq L_i$, $\beta_{\mathds{C}}(L_i)\subseteq L_i$, $\alpha_{\mathds{C}}(L_{-i})\subseteq L_{-i}$ and $\beta_{\mathds{C}}(L_{-i})\subseteq L_{-i}$.
Next, for all $X, Y, Z \in L_i$, we
have
\begin{align*}
J_{\mathds{C}}[X, Y, Z]_{L_{\mathds{C}}} =& -[J_{\mathds{C}}X, J_{\mathds{C}}Y, J_{\mathds{C}}Z]_{L_{\mathds{C}}}+ [J_{\mathds{C}}X, Y, Z]_{L_{\mathds{C}}}+[X, J_{\mathds{C}}Y, Z]_{L_{\mathds{C}}}+[X, Y, J_{\mathds{C}}Z]_{L_{\mathds{C}}}\\
&+J_{\mathds{C}}[J_{\mathds{C}}X, J_{\mathds{C}}Y, Z]_{L_{\mathds{C}}}+ J_{\mathds{C}}[X, J_{\mathds{C}}Y, J_{\mathds{C}}Z]_{L_{\mathds{C}}}+ J_{\mathds{C}}[J_{\mathds{C}}X, Y, J_{\mathds{C}}Z]_{L_{\mathds{C}}}\\
=&4i[X, Y, Z]_{L_{\mathds{C}}}-3J_{\mathds{C}}[X, Y, Z]_{L_{\mathds{C}}}.
\end{align*}
Thus, $[X, Y, Z]_{L_{\mathds{C}}} \in L_i$, which implies $L_i$ is a Bihom subalgebra. Similarly $L_{-i}$ is also a Bihom subalgebra.

Conversely, we define a complex map $J_{\mathds{C}} : L_{\mathds{C}}\rightarrow L_{\mathds{C}}$ by
\begin{equation}\label{jc}
J_{\mathds{C}}(X +\sigma(Y )) = iX-i\sigma(Y ), ~~\forall X, Y\in Q.
\end{equation}
Since $\sigma$ is a complex antilinear and involutive automorphism of $L_{\mathds{C}}$, we have
\begin{align*}
J_{\mathds{C}}^2(X+\sigma(Y )) = J_{\mathds{C}}(iX - i\sigma(Y )) = J_{\mathds{C}}(iX+ \sigma(iY )) = i(iX) - i\sigma(iY ) = -X- \sigma(Y ),
\end{align*}
i.e., $J_{\mathds{C}}^2=-\rm{Id}$. And \begin{align*}\alpha_{\mathds{C}}J_{\mathds{C}}(X+\sigma(Y ))&=\alpha_{\mathds{C}}(iX-i\sigma(Y))=i\alpha_{\mathds{C}}(X)-i\alpha_{\mathds{C}}(\sigma(Y))
=J_{\mathds{C}}(\alpha_{\mathds{C}}(X)+\alpha_{\mathds{C}}(\sigma(Y)))\\
&=J_{\mathds{C}}\alpha_{\mathds{C}}(X+\sigma(Y )),
\end{align*}
i.e., $J_{\mathds{C}}\alpha_{\mathds{C}}=\alpha_{\mathds{C}}J_{\mathds{C}}$. Similarly, $J_{\mathds{C}}\beta_{\mathds{C}}=\beta_{\mathds{C}}J_{\mathds{C}}$.
Further we can show that $J_{\mathds{C}}$ satisfies $(\ref{eq4.1})$. Since $L_{\mathds{C}}=Q\oplus P$, for all $X\in Q$ we have
\begin{align*}
J_{\mathds{C}}\sigma(X+\sigma(Y ))=J_{\mathds{C}}(Y+\sigma(X))=iY-i\sigma(X)
=\sigma(iX-i\sigma(Y))=\sigma J_{\mathds{C}}(X+\sigma(Y )),
\end{align*}
which implies that $J_{\mathds{C}}\sigma= \sigma J_{\mathds{C}}$. Moreover, since $\sigma(X ) = X$ is equivalent to $X \in L$, we know
that the set of fixed points of $\sigma$ is the real vector space $L$. By $J_{\mathds{C}}\sigma= \sigma J_{\mathds{C}}$, there is a well-defined
$J \in \rm{End}(L)$ given by
$J\triangleq J_{\mathds{C}}|_L$.
Because, $J_{\mathds{C}}^2=-\rm{Id}$, $J_{\mathds{C}}\alpha_{\mathds{C}}=\alpha_{\mathds{C}}J_{\mathds{C}}$, $J_{\mathds{C}}\beta_{\mathds{C}}=\beta_{\mathds{C}}J_{\mathds{C}}$ and $J_{\mathds{C}}$ satisfies $(\ref{eq4.1})$, $J$ is a complex structure on $L$.
\end{proof}

\begin{prop}\label{prop4.6}
Let $J$ be an almost complex structure on a real $3$-Bihom-Lie algebra $(L, [\cdot, \cdot,\cdot],\alpha,\beta)$ .
If $J$ satisfies
\begin{equation}\label{32}
J[x, y,z] = [Jx, y,z], \forall x,y,z\in L,
\end{equation}
then $J$ is a complex structure on $L$.
\end{prop}
\begin{proof}
 Since $J^2=-\rm{Id}$ and $(\ref{32})$, we have
 \begin{align*}
 &-[Jx, Jy, Jz]+[Jx, y, z]+[x, Jy, z]+[x, y, Jz]\\
 &+J[Jx, Jy, z]+J[x, Jy, Jz]+J[Jx, y, Jz]\\
=& -[Jx, Jy, Jz]+J[x, y, z]+[x, Jy, z]+[x, y, Jz]\\
&+[J^2x, Jy, z]+[Jx, Jy, Jz]+[J^2x, y, Jz]\\
= &J[x, y, z].
 \end{align*}
Thus, $J$ is a complex structure on $L$.
\end{proof}

\begin{defn}
An almost complex structure on a real $3$-Bihom-Lie algebra $(L, [\cdot, \cdot,\cdot],\alpha,\beta)$ is called
a strict complex structure if $(\ref{32})$ holds.
\end{defn}

\begin{cor}
Let $(L, [\cdot, \cdot,\cdot],\alpha,\beta)$ a real $3$-Bihom-Lie algebra with $\alpha,\beta$ surjective. Then there is a strict complex
structure on $L$ if and only if $
L_{\mathds{C}}=Q\oplus P,
$
where $Q$ and $P=\sigma(Q)$ are Bihom subalgebras of $L_{\mathds{C}}$ such that  $[Q, Q, P]_{L_{\mathds{C}}} = 0$ and $[P,P, Q]_{L_{\mathds{C}}} = 0$.
\end{cor}
\begin{proof}
Let $J$ be a strict complex structure on $L$. Then, $J_{\mathds{C}}$ is a strict
complex structure on the complex $3$-Bihom-Lie algebra $(L_{\mathds{C}}, [\cdot, \cdot,\cdot]_{L_{\mathds{C}}},\alpha_{\mathds{C}},\beta_{\mathds{C}})$. By Proposition \ref{prop4.6} and Theorem \ref{thm4}, we only need to show that $[L_i, L_i, L_{-i}]_{L_{\mathds{C}}} = 0$ and $[L_{-i},L_{-i}, L_i]_{L_{\mathds{C}}} = 0$. For all $X, Y \in L_i$ and $Z \in L_{-i}$,
on one hand we have
$$J_{\mathds{C}}[X, Y, Z]_{L_{\mathds{C}}} = [J_{\mathds{C}}X, Y, Z]_{L_{\mathds{C}}} = i[X, Y, Z]_{L_{\mathds{C}}}.$$
On the other hand, since $\alpha,\beta$ are surjective, we have $\alpha_{\mathds{C}},\beta_{\mathds{C}}$ are surjective. So there are $\tilde{X},\tilde{Y}\in L_i$ and $\tilde{Z}\in L_{-i}$ such that $X=\beta_{\mathds{C}}(\tilde{X}),Y=\beta_{\mathds{C}}(\tilde{Y})$ and $Z=\alpha_{\mathds{C}}(\tilde{Z})$. We can get
 \begin{align*}
 J_{\mathds{C}}[X, Y, Z]_{L_{\mathds{C}}}
 &= J_{\mathds{C}}[\beta_{L_{\mathds{C}}}(\tilde{X}), \beta_{L_{\mathds{C}}}(\tilde{Y}), \alpha_{L_{\mathds{C}}}(\tilde{Z})]_{L_{\mathds{C}}}
 =
 J_{\mathds{C}}[\beta_{L_{\mathds{C}}}(\tilde{Z}), \beta_{L_{\mathds{C}}}(\tilde{X}), \alpha_{L_{\mathds{C}}}(\tilde{Y})]_{L_{\mathds{C}}}\\
 &=
 [J_{\mathds{C}}\beta_{L_{\mathds{C}}}(\tilde{Z}), \beta_{L_{\mathds{C}}}(\tilde{X}), \alpha_{L_{\mathds{C}}}(\tilde{Y})]_{L_{\mathds{C}}}
 = -i[\beta_{L_{\mathds{C}}}(\tilde{Z}), \beta_{L_{\mathds{C}}}(\tilde{X}), \alpha_{L_{\mathds{C}}}(\tilde{Y})]_{L_{\mathds{C}}}\\
 &=
 -i[\beta_{L_{\mathds{C}}}(\tilde{X}), \beta_{L_{\mathds{C}}}(\tilde{Y}), \alpha_{L_{\mathds{C}}}(\tilde{Z})]_{L_{\mathds{C}}}
 =
 -i[X, Y, Z]_{L_{\mathds{C}}}.
  \end{align*}
  Thus, we obtain $[L_i, L_i, L_{-i}]_{L_{\mathds{C}}} = 0$. Similarly, $[L_{-i},L_{-i}, L_i]_{L_{\mathds{C}}} = 0$

  Conversely, we define a complex linear endomorphism $J_{\mathds{C}} : L_{\mathds{C}}\rightarrow L_{\mathds{C}}$ by $(\ref{jc})$. By the proof of Theorem \ref{thm4}, $J_{\mathds{C}}$ is an almost complex structure on $L_{\mathds{C}}$. Because $\alpha_{\mathds{C}},\beta_{\mathds{C}}$ are surjective and $[Q, Q, P]_{L_{\mathds{C}}} =[P, P, Q]_{L_{\mathds{C}}} = 0 $, we also have $[Q, P, Q]_{L_{\mathds{C}}} = [P, Q, Q]_{L_{\mathds{C}}} = [P, Q, P]_{L_{\mathds{C}}} = [Q, P, P]_{L_{\mathds{C}}} = 0$. Thus for all $X_1,X_2,X_3\in L_i$, $Y_1,Y_2,Y_3\in L_{-i}$, we can show that
  \begin{align*}
 &J_{\mathds{C}}[X_1+Y_1, X_2+Y_2, X_3+Y_3]_{L_{\mathds{C}}}\\
 =&J_{\mathds{C}}([X_1, X_2, X_3]_{L_{\mathds{C}}}+[Y_1, Y_2, Y_3]_{L_{\mathds{C}}})
 =i([X_1, X_2, X_3]_{L_{\mathds{C}}}-[Y_1, Y_2, Y_3]_{L_{\mathds{C}}})\\
 =&([X_1-Y_1, X_2+Y_2, X_3+Y_3]_{L_{\mathds{C}}})
 =[iX_1-iY_1, X_2+Y_2, X_3+Y_3]_{L_{\mathds{C}}}\\
 =&[J_{\mathds{C}}(X_1+Y_1), X_2+Y_2, X_3+Y_3]_{L_{\mathds{C}}}.
  \end{align*}
  By the proof of Theorem \ref{thm4}, we obtain that $J\triangleq J_{\mathds{C}}|_L$ is a strict complex
structure on the real $3$-Bihom-Lie algebra $(L, [\cdot, \cdot,\cdot],\alpha,\beta)$. The proof is finished.
\end{proof}

Let $J$ be an almost complex structure on a real $3$-Bihom-Lie algebra $(L, [\cdot, \cdot,\cdot],\alpha,\beta)$. We can define a
complex vector space structure on the real vector space $L$ by
\begin{equation}\label{complex}
(a + bi)x\triangleq ax + bJx, \forall a, b \in \mathds{R}, x \in L.
\end{equation}
Define two maps $\phi : L \rightarrow L_i$ and $\psi : L \rightarrow L_{-i}$ given by
$\phi(x) = \frac{1}{2}(x-iJx),
\psi(x) = \frac{1}{2}(x +iJx)$.
Clearly, $\phi$ is complex linear isomorphism and $\psi = \sigma \phi$ is a complex
antilinear isomorphism of the complex vector spaces $L$.

Let $J$ be a strict complex structure on a real $3$-Bihom-Lie algebra $(L, [\cdot, \cdot,\cdot],\alpha,\beta)$ with $\alpha,\beta$ surjective. Then with the complex
vector space structure defined above, $(L, [\cdot, \cdot,\cdot],\alpha,\beta)$ is a complex $3$-Bihom-Lie algebra. In fact, using $(\ref{32})$ and $(\ref{complex})$ we can obtain
 \begin{align*}
[(a+bi)x, y, z] &= [ax+bJx, y, z]= a[x, y, z]+b[Jx, y, z]\\
&= a[x, y, z]+bJ[x, y, z] = (a+ bi)[x, y, z].
 \end{align*}
Since $\alpha,\beta$ are surjective, the $[\cdot, \cdot,\cdot]$ is complex trilinear.

Let $J$ be a complex structure on $L$. Define a new bracket $[\cdot, \cdot, \cdot]_J : \wedge^3L \rightarrow L$by
\begin{equation}\label{J}
[x, y, z]_J=\frac{1}{4}
([x, y, z]-[x, Jy, Jz]-[Jx, y, Jz]- [Jx, Jy, z]), \forall x, y, z \in L.
\end{equation}

\begin{prop}\label{prop4.9}
Let $(L, [\cdot, \cdot,\cdot],\alpha,\beta)$ be a real $3$-Bihom-Lie algebra and $J$ be a complex structure on $L$. Then $(L, [\cdot, \cdot,\cdot]_J,\alpha,\beta)$
is a real $3$-Bihom-Lie algebra. Moreover, $J$ is a strict complex structure on $(L, [\cdot, \cdot,\cdot]_J,\alpha,\beta)$. And when $\alpha,\beta$ are surjective, the corresponding
complex $3$-Bihom-Lie algebra $(L, [\cdot, \cdot,\cdot]_J,\alpha,\beta)$ is isomorphic to the complex $3$-Bihom-Lie algebra $L_i$.
\end{prop}
\begin{proof}
First we can show that $(L, [\cdot, \cdot,\cdot]_J,\alpha,\beta)$ is a real $3$-Bihom-Lie algebra. By $(\ref{eq4.1})$, for all $x, y, z \in L$, we have
\begin{align}\label{fai}
&[\phi(x), \phi(y), \phi(z)]_{L_{\mathds{C}}} \nonumber\\
=& \frac{1
}{8}
[x -iJx, y-iJy, z-iJz]_{L_{\mathds{C}}}\nonumber\\
=&\frac{1}{8}
([x, y, z]-[x, Jy, Jz]-[Jx, y, Jz]- [Jx, Jy, z])-
\frac{1}{8}
i([x, y, Jz]+[x, Jy, z]\nonumber\\
&+[Jx, y, z]- [Jx, Jy, Jz])\nonumber\\
=&\frac{1}{8}
([x, y, z]-[x, Jy, Jz]-[Jx, y, Jz]- [Jx, Jy, z])
-
\frac{1}{8}
iJ([x, y, z]-[x, Jy, Jz]\nonumber\\
&-[Jx, y, Jz]-[Jx, Jy, z])\nonumber\\
=&\frac{1}{2}[x, y, z]_J-\frac{1}{2}iJ[x, y, z]_J\nonumber\\
=& \phi[x, y, z]_J.
\end{align}
Thus, we have $[x, y, z]_J = \varphi^{-1}[\varphi(x), \varphi(y), \varphi(z)]_{L_{\mathds{C}}}$.
In addition, we can also get $\varphi\beta=\beta_{\mathds{C}}\varphi$ and $\varphi\alpha=\alpha_{\mathds{C}}\varphi$. Since $J$ is a complex structure and by Theorem \ref{thm4}, $L_i$ is a $3$-Bihom-Lie
subalgebra. Therefore, $(L, [\cdot, \cdot,\cdot]_J,\alpha,\beta)$ is a real $3$-Bihom-Lie algebra.

Using $(\ref{eq4.1})$, for all $x, y, z \in L$, we have
\begin{align*}
J[x, y, z]_J &= \frac{1
}{4}
J([x, y, z]-[x, Jy, Jz]-[Jx, y, Jz]-[Jx, Jy, z])\\
&=
\frac{1}{4}
(-[Jx, Jy, Jz]+[Jx, y, z]+[x, Jy, z]+[x, y, Jz])\\
&= [Jx, y, z]_J,
\end{align*}
which implies that $J$ is a strict complex structure on $(L, [\cdot, \cdot,\cdot]_J,\alpha,\beta)$.

Since $\alpha,\beta$ are surjective and use the same way as above, we can obtain $(L, [\cdot, \cdot,\cdot]_J,\alpha,\beta)$ is a complex $3$-Bihom-Lie algebra. By $(\ref{fai})$, $\varphi\beta=\beta_{\mathds{C}}\varphi$ and $\varphi\alpha=\alpha_{\mathds{C}}\varphi$, we have $\varphi$ is a complex $3$-Bihom-Lie
algebra isomorphism. The proof is finished.
\end{proof}

\begin{prop}
Let $(L, [\cdot, \cdot,\cdot],\alpha,\beta)$ be a real $3$-Bihom-Lie algebra with $\alpha,\beta$ surjective and $J$ be a complex structure on $L$. Then $J$ is a
strict complex structure on $(L, [\cdot, \cdot,\cdot],\alpha,\beta)$ if and only if $[\cdot, \cdot,\cdot]_J = [\cdot, \cdot,\cdot]$.
\end{prop}
\begin{proof}
Because $J$ be a strict complex structure on $L$ and $\alpha,\beta$ are surjective, we have $J[x, y, z] = [Jx, y, z]=[x, Jy, z]=[x, y, Jz]$. So for all $x,y,x\in L$ we have \begin{align*}
[x, y, z]_J &= \frac{1
}{4}
([x, y, z]-[x, Jy, Jz]-[Jx, y, Jz]- [Jx, Jy, z])\\
& =\frac{1
}{4}
([x, y, z]-J^2[x, y, z]-J^2[x, y, z]- J^2[x, y, z])\\
&= [x, y, z].
\end{align*}

Conversely, if $[\cdot, \cdot,\cdot]_J = [\cdot, \cdot,\cdot]$, we have
$$4[x, y, z]_J =
[x, y, z]-[x, Jy, Jz]-[Jx, y, Jz]- [Jx, Jy, z]=4[x, y, z].$$
Then $[x, Jy, Jz]+[Jx, y, Jz]+[Jx, Jy, z]=-3[x, y, z]$. Thus we obtain
\begin{align*}
&4J[x, y, z]=4J[x, y, z]_J \\
=&J([x, y, z]-[x, Jy, Jz]-[Jx, y, Jz]- [Jx, Jy, z])\\
=& -[Jx, Jy, Jz]+[Jx, y, z]+[x, Jy, z]+[x, y, Jz]\\
= &3[Jx, y, z]+[Jx, y, z]\\
= &4[Jx, y, z],
\end{align*}
i.e.,$J[x, y, z]=[Jx, y, z]$. Therefore the Proposition holds.
\end{proof}

\begin{prop}\label{prop4.11}
Let $J$ be an almost complex structure on a real $3$-Bihom-Lie algebra $(L, [\cdot, \cdot,\cdot],\alpha,\beta)$.
If $J$ satisfies the following equation
\begin{equation}\label{35}
[x, y, z] = [x, Jy, Jz]+[Jx, y, Jz]+[Jx, Jy, z], \forall x,y,z\in L
\end{equation}
then $J$ is a complex structure on $L$.
\end{prop}
\begin{proof}
By $(\ref{35})$ and $J^2 =-\rm Id$, we have
\begin{align*}
&-[Jx, Jy, Jz]+[Jx, y, z]+[x, Jy, z]+[x, y, Jz]+J[Jx, Jy, z]\\
&+J[x, Jy, Jz]
+J[Jx, y, Jz]\\
=& -[Jx, J^2y, J^2z]-[J^2x, Jy, J^2z]- [J^2x, J^2y, Jz]
+[Jx, y, z]+[x, Jy, z]\\
&+[x, y, Jz]+ J[x, y, z]\\
=& J[x, y, z].
\end{align*}
Thus, $J$ is a complex structure on $L$.
\end{proof}

\begin{defn}
An almost complex structure on a real $3$-Bihom-Lie algebra $(L, [\cdot, \cdot,\cdot],\alpha,\beta)$ is called
an abelian complex structure if $(\ref{35})$ holds.
\end{defn}

\begin{re}\label{re4.13}
Let $J$ be an abelian complex structure on a real $3$-Bihom-Lie algebra $(L, [\cdot, \cdot,\cdot],\alpha,\beta)$. Then
$(L, [\cdot, \cdot,\cdot]_J,\alpha,\beta)$ is an abelian $3$-Bihom-Lie algebra.
\end{re}

\begin{cor}
Let $(L, [\cdot, \cdot,\cdot],\alpha,\beta)$ be a real $3$-Bihom-Lie algebra with $\alpha,\beta$ surjective. Then there is an abelian complex
structure on $L$ if and only if $
L_{\mathds{C}} = Q\oplus P,
$
where $Q$ and $P = \sigma(Q)$ are abelian Bihom subalgebras of $L_{\mathds{C}}$.
\end{cor}

\begin{proof}
Let $J$ be an abelian complex structure on $L$. Because $\alpha,\beta$ are surjective and by Proposition \ref{prop4.9}, we obtain that $\varphi$ is a
complex $3$-Bihom-Lie algebra isomorphism from $(L, [\cdot, \cdot,\cdot]_J,\alpha,\beta)$ to $(L_i, [\cdot, \cdot,\cdot]_{L_{\mathds{C}}},\alpha_{\mathds{C}},\beta_{\mathds{C}})$. Using Remark \ref{re4.13}, $(L, [\cdot, \cdot,\cdot]_J,\alpha,\beta)$
is an abelian $3$-Bihom-Lie algebra. Thus, $Q=L_i$ is an abelian Bihom subalgebra of $L_{\mathds{C}}$. Since $P = L_{-i} = \sigma(L_i)$,
for all $x_1+iy_1, x_2+iy_2, x_3+iy_3 \in L_i$, we have
\begin{align*}
&[\sigma(x_1+iy_1), \sigma(x_2+iy_2), \sigma(x_3+iy_3)]_{L_{\mathds{C}}}\\
=& [x_1-iy_1, x_2-iy_2, x_3-iy_3]_{L_{\mathds{C}}}\\
=& [x_1, x_2, x_3]-[x_1, y_2, y_3]- [y_1, x_2, y_3]-[y_1, y_2, x_3]
-i([x_1, x_2, y_3]+[x_1, y_2, x_3]\\
&+ [y_1, x_2, x_3]-[y_1, y_2, y_3])\\
=& \sigma([x_1, x_2, x_3]-[x_1, y_2, y_3]- [y_1, x_2, y_3]-[y_1, y_2, x_3]+i([x_1, x_2, y_3]+[x_1, y_2, x_3]\\
&+ [y_1, x_2, x_3]-[y_1, y_2, y_3]))\\
=& \sigma[x_1+iy_1, x_2+iy_2, x_3+ iy_3]_{L_{\mathds{C}}}\\
=& 0.
\end{align*}
Therefore, $P$ is an abelian Bihom subalgebra of $L_{\mathds{C}}$.

Conversely, by Theorem \ref{thm4}, $L$ has a complex structure $J$. Moreover, by Proposition \ref{prop4.9},
we have a complex $3$-Bihom-Lie algebra isomorphism $\varphi$ from $(L, [\cdot, \cdot,\cdot]_J,\alpha,\beta)$ to $(Q, [\cdot, \cdot,\cdot]_{L_{\mathds{C}}},\alpha_{\mathds{C}},\beta_{\mathds{C}})$. So $(L, [\cdot, \cdot,\cdot]_J,\alpha,\beta)$
is an abelian $3$-Bihom-Lie algebra. By the definition of $[\cdot, \cdot,\cdot]_J$, we obtain that $J$ is an abelian complex
structure on $L$. The proof is finished.
\end{proof}

\begin{prop}
Let $J$ be an almost complex structure on a real $3$-Bihom-Lie algebra $(L, [\cdot, \cdot,\cdot],\alpha,\beta)$.
If $J$ satisfies the following equation
\begin{equation}\label{38}
[x, y, z] = -J[Jx, y, z]-J[x, Jy, z]-J[x, y, Jz], \forall x,y,z\in L,
\end{equation}
then $J$ is a complex structure on $L$.
\end{prop}
\begin{proof}
By $(\ref{38})$ and $J^2 =-\rm Id$ we have
\begin{align*}
&-[Jx, Jy, Jz]+[Jx, y, z]+[x, Jy, z]+[x, y, Jz]\\
&+J[Jx, Jy, z]+J[x, Jy, Jz]+J[Jx, y, Jz]\\
= &J[J^2x, Jy, Jz]+J[Jx, J^2y, Jz]+ J[Jx, Jy, J^2z]+J[x, y, z]\\
&+J[Jx, Jy, z]+J[x, Jy, Jz]+J[Jx, y, Jz]\\
=& J[x, y, z].
\end{align*}
Thus $J$ is a complex structure on $L$.
\end{proof}

\begin{defn}
Let $(L, [\cdot, \cdot,\cdot],\alpha,\beta)$ be a real $3$-Bihom-Lie
algebra. An almost complex structure $J$ on $L$ is called a strong abelian complex structure if $(\ref{38})$ holds.
\end{defn}

\begin{cor}
Let $(L, [\cdot, \cdot,\cdot],\alpha,\beta)$ be a real $3$-Bihom-Lie algebra. Then there is a strong abelian complex
structure on $L$ if and only if $
L_{\mathds{C}} = Q\oplus P,
$
where $Q$ and $P = \sigma(Q)$ are abelian Bihom subalgebras of $L_{\mathds{C}}$ such that $[Q,Q,P]_{L_{\mathds{C}}}\subseteq Q, [Q,P,Q]_{L_{\mathds{C}}}\subseteq Q,[P,Q,Q]_{L_{\mathds{C}}}\subseteq Q,[P,P,Q]_{L_{\mathds{C}}}\subseteq P,[P,Q,P]_{L_{\mathds{C}}}\subseteq P$ and $[Q,P,P]_{L_{\mathds{C}}}\subseteq P$.
\end{cor}

\begin{prop}
Let $J$ be an almost complex structure on a real $3$-Bihom-Lie algebra $(L, [\cdot, \cdot,\cdot],\alpha,\beta)$.
If $J$ satisfies the following equation
\begin{equation}\label{39}
J[x, y, z] = -[Jx, Jy, Jz], \forall x,y,z\in L,
\end{equation}
then $J$ is a complex structure on $L$.
\end{prop}

\begin{proof}
By $(\ref{39})$ and $J^2 =-\rm Id$, we have
\begin{align*}
&-[Jx, Jy, Jz]+[Jx, y, z]+[x, Jy, z]+[x, y, Jz]+J[Jx, Jy, z]\\
&+J[x, Jy, Jz]
+J[Jx, y, Jz]\\
=& +J[x, y, z]+[Jx, y, z]+[x, Jy, z]
+[x, y, Jz]-[J^2x, J^2y, Jz]\\
&-[Jx, J^2y, J^2z]-[J^2x, Jy, J^2z]\\
=& J[x, y, z].
\end{align*}
Thus, $J$ is a complex structure on $L$.
\end{proof}

\begin{defn}
Let $(L, [\cdot, \cdot,\cdot],\alpha,\beta)$ be a real $3$-Bihom-Lie
algebra. An almost complex structure $J$ on $L$ is called a perfect complex structure if $(\ref{39})$ holds.
\end{defn}

\begin{cor}
Let $(L, [\cdot, \cdot,\cdot],\alpha,\beta)$ be a real $3$-Bihom-Lie algebra. Then there is a perfect complex
structure on $L$ if and only if $
L_{\mathds{C}} = Q\oplus P,
$
where $Q$ and $P = \sigma(Q)$ are Bihom subalgebras of $L_{\mathds{C}}$ such that $[Q,Q,P]_{L_{\mathds{C}}}\subseteq P, [Q,P,Q]_{L_{\mathds{C}}}\subseteq P,[P,Q,Q]_{L_{\mathds{C}}}\subseteq P,[P,P,Q]_{L_{\mathds{C}}}\subseteq Q,[P,Q,P]_{L_{\mathds{C}}}\subseteq Q$ and $[Q,P,P]_{L_{\mathds{C}}}\subseteq Q$.
\end{cor}

\begin{cor}
Let $J$ be a strict complex structure on a real $3$-Bihom-Lie algebra $(L, [\cdot, \cdot,\cdot],\alpha,\beta)$  with $\alpha,\beta$ surjective, Then $J$ is
a perfect complex structure on $L$.
\end{cor}

\begin{ex}\label{ex2}
Let $L$ be a $4$-dimensional vector space with respect to a basis $\{e_1, e_2, e_3 ,e_4\}$, the non-zero bracket and $\alpha,\beta$ be given by
\begin{align*}
[e_1, e_2,e_3] &= [e_1, e_3,e_2] = [e_2, e_3,e_1] = e_4,~[e_2, e_1,e_3] = [e_3, e_1,e_2] = [e_3, e_2,e_1] =- e_4,\\
[e_1, e_4,e_2] &= [e_2, e_1,e_4] = [e_2, e_4,e_1] = e_3,~[e_1, e_2,e_4] = [e_4, e_1,e_2] = [e_4, e_2,e_1] =- e_3,\\
[e_3, e_1,e_4] &= [e_3, e_4,e_1] = [e_4, e_1,e_3] = e_2,~[e_1, e_3,e_4] = [e_1, e_4,e_3] = [e_4, e_3,e_1] =- e_2,\\
[e_3, e_2,e_4] &= [e_4, e_2,e_3] = [e_4, e_3,e_2] = e_1,~[e_2, e_3,e_4] = [e_2, e_4,e_3] = [e_3, e_4,e_2] =- e_1,
\end{align*}
$$\alpha=\left(                 
  \begin{array}{cccc}   
    1 & 0 & 0 & 0\\  
    0 & -1 & 0 & 0\\
    0 & 0 & 1 & 0\\
   0 & 0 & 0 & -1
  \end{array}
\right),~\beta=\rm{Id}. $$ Then $(L, [\cdot, \cdot,\cdot], \alpha,\beta)$ is a $3$-Bihom Lie algebra.
So
\begin{equation*}
 J_1=\left(                 
  \begin{array}{cccc}   
    0 & 0 & -1 & 0\\  
    0 & 0 & 0 & 1\\
    1 & 0 & 0 & 0\\
   0 & -1 & 0 & 0
  \end{array}
\right)
and ~ J_2=\left(                 
  \begin{array}{cccc}   
    0 & 0 & -1 & 0\\  
    0 & 0 & 0 & -1\\
    1 & 0 & 0 & 0\\
   0 & 1 & 0 & 0
  \end{array}
\right)
 \end{equation*}
 are strong abelian complex structures.
 \begin{equation*} J_3=\left(                 
  \begin{array}{cccc}   
    1 & 0 & 1 & 0\\  
    0 & 1 & 0 & 1\\
    -2 & 0 & -1 & 0\\
   0 & -2 & 0 & -1
  \end{array}
\right),  J_4=\left(                 
  \begin{array}{cccc}   
    1 & 0 & -1 & 0\\  
    0 & -1 & 0 & 2\\
    2 & 0 & -1 & 0\\
   0 & -1 & 0 & 1
  \end{array}
\right) and ~ J_5=\left(                 
  \begin{array}{cccc}   
    -1 & 0 & -1 & 0\\  
    0 & 1 & 0 & 2\\
    2 & 0 & 1 & 0\\
   0 & -1 & 0 & -1
  \end{array}
\right)
\end{equation*}
are abelian complex structures.
\end{ex}

At the end of this section, we give a condition between a complex structure and a product
structure on a $3$-Bihom-Lie algebra to introduce the definition of a complex product structure. And we show the relation between a
complex structure and a product structure on a complex $3$-Bihom-Lie algebra.

\begin{prop}\label{4.23}
Let $(L, [\cdot, \cdot,\cdot],\alpha,\beta)$ be a complex $3$-Bihom-Lie algebra. Then  $E$ is a product structure on $L$
if and only if $ J = iE$ is a complex structure on $L$.
\end{prop}

\begin{proof}
Let $E$ be a product structure on $L$. Then we have $J^2 = i^2E^2 =-\rm Id$, $J\alpha=iE\alpha=\alpha iE=\alpha J$ and $J\beta=\beta J$, i.e., $J$ is an almost
complex structure on $L$. Moreover, by $(\ref{21})$ we can get
\begin{align*}
&J[x, y, z] = iE[x, y, z]\\
=&i([Ex, Ey, Ez]+[Ex, y, z]+[x, Ey, z]+[x, y, Ez]-E[Ex, Ey, z]\\
&-E[x, Ey, Ez]-E[Ex, y, Ez])\\
= &-[iEx, iEy, iEz]+[iEx, y, z]+[x, iEy, z]+[x, y, iEz]+iE[iEx, iEy, z]\\
&+iE[x, iEy, iEz]+iE[iEx, y, iEz]\\
=&-[Jx, Jy, Jz]+[Jx, y, z]+[x, Jy, z]+[x, y, Jz]+J[Jx, Jy, z]\\
&+J[x, Jy, Jz]+J[Jx, y, Jz].
\end{align*}
Thus, $J$ is a complex structure on $L$.

Conversely, since $ J = iE$, we obtain $ E = -iJ$. So using the same way as above, we have $E$ is a product structure on $L$.
\end{proof}

\begin{defn}
Let $(L, [\cdot, \cdot,\cdot],\alpha,\beta)$ be a real $3$-Bihom-Lie algebra. A complex product structure on $L$ is a pair $(J, E)$ consisting of a complex structure $J$ and a product structure $E$ such that $JE =-E J$.
\end{defn}

\begin{re}\label{re4}
Let $(J, E)$ be a complex product structure on a real $3$-Bihom-Lie algebra $(L, [\cdot, \cdot,\cdot],\alpha,\beta)$. For all
$x \in L_+$, we have $E(Jx) =-J(Ex)= -Jx$, which implies that $J(L_+)\subset L_-$. Clearly, we also can obtain
$J(L_-)\subset L_+$. Thus, we get $J(L_-)= L_+$ and $J(L_+)= L_-$.
\end{re}

\begin{thm}
Let $(L, [\cdot, \cdot,\cdot],\alpha,\beta)$ be a real $3$-Bihom-Lie algebra. Then $L$ has a complex product structure if and only if $L$ has a complex structure $J$ and
$
L=L_+\oplus L_-,
$
where $L_-$ and $L_+$ are Bihom subalgebras of $L$ and $J(L_+)= L_-$.
\end{thm}
\begin{proof}
Let $(J, E)$ be a complex product structure and $L_{\pm1}$ be the eigenspaces corresponding
to the eigenvalues $\pm1$ of $E$. By Theorem \ref{thm1}, $
L=L_+\oplus L_-
$ and $L_-$, $L_+$ are Bihom subalgebras of $L$. And by Remark \ref{re4} we have $J(L_+)= L_-$.

Conversely, we can define a linear map $E : L \rightarrow L$ by
$$E(x +y) = x-y, \forall x \in L_+, y\in L_-.$$
By Theorem \ref{thm1}, $E$ is a product structure on $L$. By $J(L_+)= L_-$ and $J^2 = -\rm{Id}$, we have $J(L_-)= L_+$. So $\forall x \in L_+, y\in L_-$,
$$E(J(x+y)) = E(J(x)+J(y)) = -J(x)+ J(y) = -J(E(x+y)),$$ i.e., $EJ=-JE$.
Thus, $(J, E)$ is a complex product structure on $L$.
\end{proof}

\begin{ex}
Consider the complex structures on the 4-dimensional $3$-Bihom-Lie algebra in Example \ref{ex2}. We have \begin{equation*}       
E_1=\left(                 
  \begin{array}{cccc}   
    1 & 0 & 0 & 0\\  
    0 & 1 & 0 & 0\\
    0 & 0 & -1 & 0\\
    0 & 0 & 0 & -1
  \end{array}
\right),
E_2=\left(                 
  \begin{array}{cccc}   
    1 & 0 & 0 & 0\\  
    0 & -1 & 0 & 0\\
    0 & 0 & 1 & 0\\
    0 & 0 & 0 & -1
  \end{array}
\right),
E_3=\left(                 
  \begin{array}{cccc}   
    1 & 0 & 0 & 0\\  
    0 & -1 & 0 & 0\\
    0 & 0 & -1 & 0\\
    0 & 0 & 0 & 1
  \end{array}
\right)                 
\end{equation*}
are perfect and abelian product structures. Then $(J_1, E_1)$, $(J_2, E_1)$, $(J_1, E_3)$ and $(J_2, E_3)$, are complex product structures.

\end{ex}


\begin{thebibliography}{99}

\bibitem{KAA} K. Abdaoui, A. Ben Hassine, A. Makhlouf. BiHom-Lie colour algebras structures. arXiv:1706.02188.
\bibitem{AA}  A. Andrada. Complex product structures on 6-dimensional nilpotent Lie algebras. Forum Math. 20 (2008), no. 2, 285-315.



\bibitem{AB1}  A. Andrada, M. Barberis, I. Dotti. Classification of abelian complex structures on
6-dimensional Lie algebras. J. Lond. Math. Soc. (2) 83 (2011), no. 1, 232-255.

\bibitem{AB2}  A. Andrada, M. Barberis, I. Dotti, G. Ovando. Product structures on four dimensional
solvable Lie algebras. Homology Homotopy Appl. 7 (2005), no. 1, 9-37.

\bibitem{AS} A. Andrada, S. Salamon. Complex product structures on Lie algebras.  Forum Math. 17 (2005), no. 2, 261-295.

\bibitem{BN} A. Ben Hassine, S. Mabrouk, O. Ncib. $3$-BiHom-Lie superalgebras induced by BiHom-Lie superalgebras. Linear Multilinear Algebra. doi: 10.1080/03081087.2020.1713040.
\bibitem {CQ} Y. Cheng, H. Qi. Representations of Bihom-Lie algebras. arXiv:1610.04302.


\bibitem{GMMP} G. Graziani, A. Makhlouf, C. Menini, F. Panaite. BiHom-associative algebras, BiHom-Lie algebras and BiHom-bialgebras. SIGMA Symmetry Integrability Geom. Methods Appl. 11 (2015), Paper 086, 34 pp.
\bibitem {AAS} A. Kitouni, A. Makhlouf, S. Silvestrov. On $n$-ary Generalization of BiHom-Lie Algebras and BiHom-Associative Algebras. arXiv:1812.00094.
\bibitem{LC} J. Li, L. Chen. The construction of $3$-Bihom-Lie algebras. Comm. Algebra 48  (2020), no. 12, 5374-5390.
\bibitem{JLY} J. Li, L. Chen, Y. Cheng. Representations of Bihom-Lie superalgebras. Linear Multilinear Algebra 67 (2019), no. 2, 299-326.

\bibitem{LMC} L. Liu, A. Makhlouf, C. Menini, F. Panaite. BiHom-Novikov algebras and infinitesimal BiHom-bialgebras. J. Algebra 560 (2020), 1146-1172.
\bibitem{QLJ} Y. Qin, J. Lin and J. Jiang. Product and complex structures on Hom-$3$-Lie algebra. doi: 10.13140/RG.2.2.27140.07041.
\bibitem{SS} S. Salamon. Complex structures on nilpotent Lie algebras. J. Pure Appl. Algebra 157 (2001), no. 2-3, 311-333.
\bibitem{ST} Y. Sheng and R. Tang. Symplectic, product and complex structures on 3-Lie algebras. J.
Algebra, 508 (2018), 256-300.






\end{thebibliography}
\end{document}